\documentclass[12pt]{article}
\usepackage{amsthm,amssymb,amsmath,graphicx,cite}
\usepackage{algorithmic, algorithm}
\usepackage[margin = 1.15in]{geometry}
\usepackage{multirow}
\usepackage{float}
\usepackage{fullwidth}
\usepackage{placeins}
\usepackage[margin=1cm]{caption}
\usepackage{makecell}
\usepackage{multirow}
\usepackage{float}
\setcellgapes{4pt}
\usepackage{subfig}
\usepackage{color}

\newcommand{\relint}{\text{relint}}

\newcommand{\Oh}{{\mathcal O}}
\newcommand{\dmin}{\displaystyle\min}
\newcommand{\dmax}{\displaystyle\max}

\newcommand{\ip}[2]{\left\langle #1 , #2 \right\rangle}    % inner product
\newcommand{\R}{{\mathbb R}}

\newcommand{\E}{{\mathbb E}}
\newcommand{\F}{{\mathbb F}}

\newcommand{\dom}{{\mathrm{dom}}}

\newcommand{\D}{{\mathcal{D}}}
\DeclareMathOperator*{\argmin}{arg\,min}

\newtheorem{lemma}{Lemma}
\newtheorem{theorem}{Theorem}
\newtheorem{proposition}{Proposition}

\newtheorem{corollary}{Corollary}

\newtheorem{definition}{Definition}
\newtheorem{assumption}{Assumption}

\newcommand{\gcggap}{{\mathsf{CGgap}}}

\title{Perturbed Fenchel duality and first-order  methods}
\author{
David H. Gutman\thanks{Department of Industrial, Manufacturing and Systems Engineering,
Texas Tech University, USA, {\tt David.Gutman@ttu.edu}}
 \and
Javier F. Pe\~na\thanks{Tepper School of Business,
Carnegie Mellon University, USA, {\tt jfp@andrew.cmu.edu}}
}

\begin{document}

\maketitle

\begin{abstract}

We show that the iterates generated by a generic {\em first-order meta-algorithm}  satisfy a canonical {\em perturbed} Fenchel duality inequality.  The latter in turn readily yields a unified derivation  of the best known convergence rates for various popular first-order algorithms including the conditional gradient method as well as the main kinds of Bregman proximal methods: subgradient, gradient, fast gradient, and universal gradient methods.

\end{abstract}
\section{Introduction}

The central goal of this paper is to highlight an interesting connection between two major threads in convex optimization, namely {\em first-order methods} and {\em Fenchel duality}.  More precisely, we show that a generic {\em first-order meta-algorithm} generates iterates that satisfy a canonical {\em perturbed} Fenchel duality property.  The meta-algorithm relies on some {\em flexibly selected sequence} of points $\{y_k: k=0,1,\dots\}$.  Specific choices of the sequence $\{y_k: k=0,1,\dots\}$ turn our first-order meta-algorithm into each of the following popular first-order methods: conditional gradient, Bregman proximal gradient, Bregman proximal subgradient, fast Bregman proximal gradient, and universal Bregman proximal gradient.
The perturbed Fenchel duality property in turn yields a generic convergence result for the meta-algorithm and thereby a unified approach to derive the best known convergence rates for all of the above popular first-order algorithms.

Each of our two main results (Theorem~\ref{thm.bregman.subgrad} and Theorem~\ref{thm.bregman.grad}) provides a duality relationship between a {\em primal average sequence} and a {\em dual average sequence} of the iterates generated by our first-order meta-algorithm.  These kinds of average sequences underlie (sometimes implicitly) the construction and analyses of  several first-order algorithms, prime examples being the analysis of the mirror descent method~\cite{NemiY83} and the dual averaging method of Nesterov~\cite{Nest05}.  It is noteworthy that Theorem~\ref{thm.bregman.subgrad} and Theorem~\ref{thm.bregman.grad} do not depend on any assumptions concerning Lipschitz continuity or strong convexity. Indeed, our main developments in Section~\ref{sec.bregman} do not rely on any norms at all.
We build on and extend ideas introduced in~\cite{GutmP18,Pena17}.  However, unlike these predecessors that apply only to Euclidean proximal methods, the main results in this paper (Theorem~\ref{thm.bregman.subgrad} and Theorem~\ref{thm.bregman.grad}) apply in the more general context of Bregman proximal methods and hold under minimal assumptions.

Our work is related to a variety of recent articles that aim to add intuition,  explain, and unify the convergence of first-order methods.  Some of these approaches are based on continuous-time dynamics~\cite{DiakO19,SuBC14}, control theory~\cite{LessRP16}, online optimization~\cite{AberW17,Aber18,WangA18}, and geometric~\cite{BubeLS15,DrusFR16}  techniques.  In contrast to these approaches, our results apply in the broad context of Bregman proximal methods as developed in the recent articles~\cite{BausBT16,HanzRX18,LuFN18,Tebo18, Drago19} and are predicated upon a perturbed Fenchel duality property.

Our developments facilitate a unified analysis of not only the Bregman gradient and subgradient schemes as discussed in~\cite{Tebo18}, but also the universal proximal gradient scheme~\cite{Nest15}, the recent accelerated Bregman method of~\cite{HanzRX18}, and the conditional gradient method~\cite{Bach15,Jagg13}.  The conditional gradient method corresponds to a special and particularly simple case of our first-order meta-algorithm by setting a suitable reference function identically to zero.  This simple connection between the conditional gradient method and Bregman proximal methods appears to have been previously overlooked. 

 As a nice byproduct of our approach, we readily obtain convergence properties for the conditional gradient method that are novel and stronger along several dimensions.  First, we consider a composite minimization problem as the one considered in~\cite{Nest18} but without any boundedness assumption on the domain.   The composite format includes as a special case the common format of convex minimization over a convex compact domain that is typically discussed in the literature on the conditional gradient method~\cite{Bach15,Beck17,FreuG16,Jagg13}.  Second, we provide a bound on the {\em Fenchel duality gap} between the main (primal) iterates and some canonical dual iterates.  Our bound applies to a composite objective and thus is more general than the kind of duality gap bound discussed in~\cite{FreuG16,Jagg13} that only applies to smooth convex minimization over a compact convex domain. Third, we derive a rate of convergence that relies on a  {\em curvature condition} that generalizes the {\em curvature constant} introduced by Jaggi~\cite{Jagg13}.   Fourth, our approach naturally suggests a line-search strategy for selecting the stepsize at each main iteration.  This line-search strategy yields a novel linear convergence result on the duality gap when the components of the composite objective satisfy a type of joint smoothness and strong convexity assumption.

\subsection{Main contribution: perturbed Fenchel duality}

Consider the convex minimization problem
\begin{equation}\label{eq.primal}
\dmin_{x\in \E}  f(Ax) + \Psi(x)
\end{equation}
and its Fenchel dual
\begin{equation}\label{eq.dual}
\dmax_{u\in \F^*}  -f^*(u) - \Psi^*(-A^*u),
\end{equation}
where $\E$ and $\F$ are real vector spaces, $A:\E\rightarrow \F$ is a linear mapping, and $f:\F\rightarrow \R\cup\{\infty\}$  and $\Psi: \E\rightarrow \R\cup\{\infty\}$ are closed convex functions. The spaces $\E^*$ and $\F^*$ are the continuous dual vector spaces of $\E$ and $\F$, and the functions $f^*:\F^*\rightarrow \R\cup\{\infty\}$ and $\Psi^*: \E^*\rightarrow \R\cup\{\infty\}$ are the convex conjugates of $f$ and $\Psi$ respectively, namely,
\[
f^*(u) = \sup_{y\in \F} \{\ip{u}{y} - f(y)\},\; \; 
\Psi^*(v) = \sup_{x\in \E} \{\ip{v}{x} - f(x)\}.
\]

It is known and easy to see, as detailed  in~\cite{BausC11,BorwL00,HiriL93,Rock70}, that {\em weak duality} holds between the primal and dual problems~\eqref{eq.primal} and~\eqref{eq.dual}, that is, for all $x\in \E$ and $u\in \F^*$
\begin{equation}\label{eq.weak.dual}
f(Ax) + \Psi(x) \ge  -f^*(u) - \Psi^*(-A^*u).
\end{equation}

From~\eqref{eq.weak.dual} it readily follows that $\bar x \in \E$ and $\bar u\in \F^*$ are  optimal solutions to~\eqref{eq.primal} and~\eqref{eq.dual} respectively if the following Fenchel duality identity holds
\begin{equation}\label{eq.duality}
f(A\bar x) + \Psi(\bar x) + f^*(\bar u) + \Psi^*(-A^*\bar u) = 0.
\end{equation}

The central contributions of this paper (Theorem~\ref{thm.bregman.subgrad} and Theorem~\ref{thm.bregman.grad}) show that a generic {\em first-order meta-algorithm} (Algorithm \ref{algo.bregman}), which includes a wide class of first-order methods, generates {\em average sequences} $x_k \in \E$ and $u_k\in \F^*$ that satisfy the following perturbed version of \eqref{eq.duality} for some $\delta_k \ge 0$ and some $d_k:\E\rightarrow \R_+\cup\{\infty\}$ 
\begin{equation}\label{eq.pert.duality}
f(Ax_k) + \Psi(x_k) + f^*(u_k) + (\Psi+d_k)^*(-A^*u_k) \le \delta_k.
\end{equation}

Inequality~\eqref{eq.pert.duality} can be considered a type of {\em perturbed Fenchel duality} property.
This property automatically yields the generic convergence results~\eqref{eq.gap} and~\eqref{eq.pert.duality.again} below.  Observe that 
for all $x\in \E$ we have 
\begin{align*}
f^*(u_k) + (\Psi+d_k)^*(-A^*u_k) &\ge \ip{u_k}{Ax} - f(Ax) + \ip{-A^*u_k}{x} - \Psi(x) - d_k(x) \\
&=  - f(Ax)  - \Psi(x) - d_k(x).
\end{align*}
Therefore the perturbed Fenchel duality property~\eqref{eq.pert.duality} implies that for all $x\in \E$
\begin{equation}\label{eq.gap}
f(Ax_k) + \Psi(x_k) - (f(Ax) + \Psi(x))\le  d_k(x) + \delta_k
\end{equation}
and so $f(Ax_k) + \Psi(x_k)$ converges to $\min_{x\in\E}\{f(Ax)+\Psi(x)\}$ provided both $d_k$ and $\delta_k$ converge to zero.

\medskip

The perturbed Fenchel duality property~\eqref{eq.pert.duality} also yields the following duality gap bound on a sequence of perturbed problems. By adding $d_k(x_k)$ to both sides of~\eqref{eq.pert.duality}, we can rewrite~\eqref{eq.pert.duality} as follows
\begin{equation}\label{eq.pert.duality.again}
f(Ax_k) + (\Psi+d_k)(x_k) + f^*(u_k) + (\Psi+d_k)^*(-A^*u_k) \le d_k(x_k) + \delta_k.
\end{equation}
This gives a generic duality gap bound for the following perturbation of the original problem~\eqref{eq.primal} 
\[
\min_{x\in \E} f(Ax) + (\Psi+d_k)(x)
\]
and its dual
\[
\max_{u\in \F^*} -f^*(u) - (\Psi+d_k)^*(-A^*u).
\]
The duality gap bound~\eqref{eq.pert.duality.again} automatically implies that both $f(Ax_k)+\Psi(x_k)$ and $-f^*(u_k) - \Psi^*(-A^*u_k)$ converge to  $\min_{x\in\E}\{f(Ax)+\Psi(x)\} = \max_{u\in\F^*}\{-f^*(u) - \Psi^*(-A^*u)\}$ provided both $d_k$ and $\delta_k$ converge to zero.

\medskip

  Theorem~\ref{thm.bregman.subgrad} and Theorem~\ref{thm.bregman.grad} exhibit explicit expressions for the objects $x_k\in \E, u_k\in \E^*, \delta_k \in \R_+,$ and $d_k:\E\rightarrow \R_+\cup\{\infty\}$ that appear in~\eqref{eq.pert.duality} in terms of the steps that the meta-algorithm performs. 
More precisely, equation~\eqref{eq.dk} shows that 
\[
d_k(x) = \frac{D_h(x,x_0)}{\sum_{i=0}^{k-1} t_i}
\]
where $D_h$ is the Bregman distance generated by some reference function $h$, $x_0$ is the initial point for the meta-algorithm, and $\{t_i:i=0,1,\dots\}$ are stepsizes that the meta-algorithm chooses at each main iteration.  Equations~\eqref{eq.pertdual.subgrad} in Theorem~\ref{thm.bregman.subgrad} and~\eqref{eq.pertdual.grad} in Theorem~\ref{thm.bregman.grad} provide a similar expression for $\delta_k$ as a fraction where the denominator is also $\sum_{i=0}^{k-1} t_i$ and the numerator depends on both $\{t_i:i=0,1,\dots\}$ and $\{y_i: i=0,1,\dots\}.$  Therefore the following key insight is readily revealed: 
both $d_k$ and $\delta_k$ in~\eqref{eq.pert.duality} can be driven to zero provided the meta-algorithm makes judicious choices of  $\{t_i: i=0,1,\dots\}$ and $\{y_i: i =0,1,\dots\}$ so that the denominator $\sum_{i=0}^{k-1} t_i$  grows sufficiently faster than the numerator in each of the fractions
in~\eqref{eq.dk},~\eqref{eq.pertdual.subgrad}, and~\eqref{eq.pertdual.grad}.

\medskip

As detailed in Section~\ref{sec.algos}, we leverage the perturbed Fenchel duality results (Theorem~\ref{thm.bregman.subgrad} and Theorem~\ref{thm.bregman.grad}) to provide  a unified derivation of the best known convergence rates of the most popular first-order algorithms,
 namely:  the conditional gradient method~\cite{Bach15,Jagg13}, the Bregman proximal gradient and subgradient algorithms~\cite{BausBT16,BeckT03,Lu17,LuFN18,Tebo18}, and the fast and universal Bregman proximal gradient algorithms~\cite{BeckT09,HanzRX18,Nest13,Nest15}.  Each of these algorithms corresponds to a special case of our first-order meta-algorithm for some specific choice of the sequence $\{y_k:k=0,1,\dots\}$.  Furthermore, the convergence rate of each of these algorithms follows from the above perturbed Fenchel duality approach by driving $d_k$ and $\delta_k$ to zero in some particular way.  
 
 Section~\ref{sec.condgrad} discusses a generalized conditional subgradient method.  This is a special case of our first-order meta-algorithm where $h\equiv 0$ and thus $d_k \equiv 0$.  For this method the stepsizes can be chosen so that $\delta_k = \Oh(1/k^\nu)$ and thus the duality gap converges to zero at the rate $\Oh(1/k^\nu)$ provided a suitable curvature condition~\eqref{eq.curv.nu} holds for some $\nu > 0$ (see  Proposition~\ref{prop.condgrad}) and a suitable choice of stepsizes.   The popular convergence rate $\Oh(1/k)$ of the conditional gradient method discussed in~\cite{Beck17,FreuG16,Jagg13} can be recovered as a special case of Proposition~\ref{prop.condgrad} when $\nu=1$. Most interestingly,  when a suitable joint smoothness and strong convexity condition~\eqref{eq.smooth.conv} holds, we show that the duality gap converges linearly to zero for properly chosen stepsizes (see Proposition~\ref{prop.condgrad.linear}).

Section~\ref{sec.grad} discusses the Bregman gradient method.  For this method the stepsizes can be chosen so that both $1/\sum_{i=0}^{k-1} t_i = \Oh(1/k)$ and $\delta_k = 0$ provided a suitable relative smoothness condition~\eqref{eq.smooth.1} holds.  The iconic $\Oh(1/k)$ convergence rate from~\cite{BausBT16,LuFN18,Tebo18} thus readily follows (see  Proposition~\ref{prop.grad}).  Section~\ref{sec.subgrad} discusses the Bregman subgradient method.  For this method we show that $\delta_k$ can be bounded above by an expression of the form $\Oh\left(\sum_{i=0}^{k-1} t_i^2/\sum_{i=0}^{k-1} t_i\right)$ provided a suitable relative continuity condition~\eqref{eq.relcont} holds (see Proposition~\ref{prop.subgrad}).  The iconic $\Oh(1/\sqrt{k})$ convergence rate of the subgradient method~\cite{Lu17,Tebo18} thus follows by choosing the first $k$ stepsizes $t_0,\dots,t_{k-1}$ all equal to a constant multiple of $1/\sqrt{k}$.   

Section~\ref{sec.fast.grad} discusses the fast Bregman gradient method.  In this case the stepsizes can be chosen so that both $1/\sum_{i=0}^{k-1} t_i = \Oh(1/k^\gamma)$ and $\delta_k = 0$ provided a suitable relative smoothness condition~\eqref{eq.smooth.2} holds for some $\gamma > 1$ (see Proposition~\ref{prop.fast}).  The iconic $\Oh(1/k^2)$ convergence rate of fast proximal gradient methods~\cite{BeckT09,Nest83,Nest13} as well as the more recent related results in~\cite{HanzRX18} can be obtained as a special case of 
Proposition~\ref{prop.fast}.  Section~\ref{sec.universal} discusses the universal Bregman gradient method.  In this case for any $\epsilon > 0$ the stepsizes can be chosen so that both $1/\sum_{i=0}^{k-1} t_i = \Oh(1/\epsilon^{\frac{1-\nu}{1+\nu}}k^{\frac{1+3\nu}{1+\nu}})$ and $\delta_k = \epsilon$ provided a suitable relative smoothness condition~\eqref{eq.smooth.3} holds for some $\nu \in [0,1]$ (Proposition~\ref{prop.univ}).  The optimal $\Oh(1/k^{\frac{1+3\nu}{2}})$ universal convergence rate for minimization of smooth convex functions with $\nu$-H\"older continuous gradient~\cite{Nest15} follows as a special case of Proposition~\ref{prop.univ}.

\subsection{Organization of the paper}

The main sections of the paper are organized as follows.  Section \ref{sec.bregman} presents our central developments.  This section describes a first-order meta-algorithm (Algorithm \ref{algo.bregman}) and its perturbed Fenchel duality properties (Theorem~\ref{thm.bregman.subgrad} and Theorem~\ref{thm.bregman.grad}).  Section~\ref{sec.bregman} also states a generic convergence result for Algorithm~\ref{algo.bregman} (Corollary~\ref{cor.conv.generic}).  Section \ref{sec.algos} leverages the perturbed Fenchel duality properties of Section~\ref{sec.bregman} to provide succinct and unified derivations of the best known convergence rates for a variety of popular first-order algorithms. Section \ref{sec.proofs} presents the proofs of our main results, Theorem~\ref{thm.bregman.subgrad} and Theorem~\ref{thm.bregman.grad}.

Our exposition is self-contained.  We rely primarily on standard convex analysis techniques as discussed in~\cite{BausC11,BorwL00,HiriL93,Rock70}. 
The seed for this paper was the technical report~\cite{GutmP18b}, which can be considered an initial draft of this work.  The exposition and main ideas first laid out in~\cite{GutmP18b} have been substantially re-organized, cleaned up, and extended.

\section{First-order meta-algorithm}
\label{sec.bregman}

Suppose $\E$ and $\F$ are real vector spaces, $A:\E \rightarrow \F$ is a linear mapping, and $f:\F\rightarrow \R\cup\{\infty\}$  and $\Psi: \E\rightarrow \R\cup\{\infty\}$ are closed convex functions such that $\dom(\Psi) \subseteq \dom(\partial (f\circ A))$. Furthermore, suppose $h:\E\rightarrow \R\cup\{\infty\}$ is a convex differentiable {\em reference function} such that $\dom(\Psi) \subseteq \overline{\dom(h)}$.
Algorithm~\ref{algo.bregman} describes a first-order meta-algorithm for problem~\eqref{eq.primal}.

The key step of Algorithm~\ref{algo.bregman}, namely Step 5, hinges on the {\em Bregman proximal mapping}
\begin{equation}\label{eq.prox.map}
(s_-,g) \mapsto \argmin_s\{t (\ip{A^*g}{s} + \Psi(s)) + D_h(s,s_{-})\}.
\end{equation}
Here $D_h$ is the {\em Bregman} distance generated by the reference function $h$, that is,
\[
D_h(s,z) = h(s) - h(z) - \ip{\nabla h(z)}{s-z}.
\]
We note that for additional flexibility, we relax the common strict convexity requirement of the reference function $h$. This flexibility enables us to use linear oracles to solve~\eqref{eq.prox.map} when $h\equiv 0$ and thus includes the conditional gradient method in our framework.  Since $h$ is not required to be strictly convex, the minimizer of the problem~\eqref{eq.prox.map} may not be unique.  Step 5 of Algorithm~\ref{algo.bregman} does not require uniqueness of the solution to~\eqref{eq.prox.map}.  Instead it only relies on the assumption that {\em a solution} to~\eqref{eq.prox.map} is computable.  Algorithm~\ref{algo.bregman} also relies on a flexibly selected sequence $\{y_k:k=0,1,\dots\}$.  This sequence in turn yields the input $g_k\in\partial f(Ay_k)$ used to generate the main sequence of iterates $\{s_k: k=0,1,\dots\}$ via~\eqref{eq.prox.map}.  Section~\ref{sec.algos} below details how specific choices of $\{y_k:k=0,1,\dots\}$ turn Algorithm~\ref{algo.bregman} into each of the following popular first-order methods: conditional gradient, Bregman proximal gradient and subgradient, and Bregman fast and universal proximal gradient.

\begin{algorithm}
\caption{First-order meta-algorithm}\label{algo.bregman}
\begin{algorithmic}[1]
	\STATE {\bf input:}  $(A,f,\Psi)$ and $s_{-1} \in \dom(\Psi) \cap \dom(h)$
	\FOR{$k=0,1,2,\dots$}
		\STATE pick $y_{k} \in \dom(\partial (f\circ A))$
%		\STATE pick 
		and $g_k \in \partial f(Ay_k)$
		\STATE pick $t_k > 0$ 
		\STATE pick $s_{k} \in \argmin_s\{t_k(\ip{A^*g_k}{s} + \Psi(s)) + D_h(s,s_{k-1})\}$
	\ENDFOR
\end{algorithmic}
\end{algorithm}

Algorithm~\ref{algo.bregman} naturally relies on the assumption that each instance of~\eqref{eq.prox.map} in Step 5 is well-posed and computable.  Indeed, we will make the following blanket
assumption throughout the paper.

\begin{assumption}\label{assump.blanket} At each iteration $k=0,1,\dots$ of Algorithm~\ref{algo.bregman}  Step 5 is well-posed, that is, a minimizer $s_{k} \in \argmin_s\{t_k(\ip{A^*g_k}{s} + \Psi(s)) + D_h(s,s_{k-1})\}$ exists, is computable, and satisfies the optimality conditions
\begin{equation}\label{eq.op.conds}
t_k(A^*g_k + g_k^\Psi) + \nabla h(s_{k}) - \nabla h(s_{k-1}) = 0
\end{equation}
for some $g^\Psi_k\in \partial\Psi(s_{k})$.
\end{assumption}

We should highlight that  the subgradient $g^\Psi_k$ in Assumption~\ref{assump.blanket} can be interpreted as a certificate of optimality for the solution $s_k \in \argmin_s\{t_k(\ip{A^*g_k}{s} + \Psi(s)) + D_h(s,s_{k-1})\}$.  This certificate is used in the construction of the sequence $w_k\in \E^*$ in equation~\eqref{eq.dk} below.  This sequence is a key component of our main results.

\medskip

Our main results (Theorem~\ref{thm.bregman.subgrad} and
Theorem~\ref{thm.bregman.grad}) show that  Algorithm~\ref{algo.bregman} yields average sequences of iterates that satisfy a perturbed Fenchel duality inequality of the form~\eqref{eq.pert.duality}.  It is worthwhile noting that these two results hold under the minimal assumptions stated above.  As we subsequently detail in Section~\ref{sec.algos}, several popular first-order algorithms including the conditional gradient, Bregman proximal gradient and subgradient algorithms, and fast and universal Bregman proximal gradient algorithms, can all be seen as special cases of Algorithm~\ref{algo.bregman}.  Most interesting, their convergence properties readily follow  the from the perturbed Fenchel results stated in Theorem~\ref{thm.bregman.subgrad} and
Theorem~\ref{thm.bregman.grad}.

Our perturbed Fenchel duality results concern the {\em primal average} sequences $z_k, x_k \in \E$, and the {\em dual average} sequence $u_k\in \F^*$ defined by Algorithm~\ref{algo.bregman} as follows
\begin{equation}\label{eq.xseq}
 x_k := 
\frac{\sum_{i=0}^{k-1} t_is_{i}}{\sum_{i=0}^{k-1} t_i}, \;
z_k := 
\frac{\sum_{i=0}^{k-1} t_iy_{i}}{\sum_{i=0}^{k-1} t_i}, \;\; 
u_k := \frac{\sum_{i=0}^{k-1} t_i g_i}{\sum_{i=0}^{k-1} t_i}.
\end{equation}
For notational convenience we let $x_0:=z_0 := s_{-1}$.
Our perturbed Fenchel duality results also rely on the sequences  $d_k:\E\rightarrow \R\cup\{\infty\}$ and $w_k\in \E^*$ defined as follows
\begin{equation}\label{eq.dk}
d_k(s):=\frac{D_h(s,s_{-1})}{\sum_{i=0}^{k-1} t_i}, \; w_k := \frac{\sum_{i=0}^{k-1} t_i(A^*g_i+g_i^\Psi)}{\sum_{i=0}^{k-1} t_i}.
\end{equation}

\begin{theorem}\label{thm.bregman.subgrad}
The iterates generated by Algorithm~\ref{algo.bregman} satisfy
\begin{multline}\label{eq.subgrad}
\frac{\sum_{i=0}^{k-1} t_i(f(Ay_{i}) + \Psi(y_{i}) + f^*(g_i) + \Psi^*(g_i^\Psi))
}{\sum_{i=0}^{k-1} t_i} + d_k^*(-w_k)
\\ 
=  \frac{\sum_{i=0}^{k-1}  t_i(\Psi(y_{i}) - \Psi(s_{i}) - \ip{A^*g_i}{s_{i}-y_{i}}) - D_h(s_{i},s_{i-1})}{\sum_{i=0}^{k-1} t_i}.
\end{multline}
In particular, the following perturbed Fenchel duality property holds
\begin{equation}\label{eq.pert.dual.2}
f(Az_k) + \Psi(z_k) + f^*(u_k) + (\Psi+d_k)^*(-A^*u_k) \le \delta_k
\end{equation}
for
\begin{equation}\label{eq.pertdual.subgrad}
\delta_k =  \frac{\sum_{i=0}^{k-1}  t_i(\Psi(y_{i}) - \Psi(s_{i}) - \ip{A^*g_i}{s_{i}-y_{i}}) - D_h(s_{i},s_{i-1})}{\sum_{i=0}^{k-1} t_i}.
\end{equation}
\end{theorem}

As we detail in Section~\ref{sec.algos} below, Theorem~\ref{thm.bregman.subgrad} yields the known convergence rates of the proximal subgradient method~\cite{Bell17,Lu17,Tebo18}.  Theorem~\ref{thm.bregman.subgrad} also yields our second  perturbed  Fenchel duality, namely Theorem~\ref{thm.bregman.grad} below.  The statement of Theorem~\ref{thm.bregman.grad} relies on two more pieces of notation. First, let $\theta_k \in [0,1], \; k=0,1,\dots$ be defined as follows
\begin{equation}\label{eq.theta}
\theta_k:= \frac{t_k}{\sum_{i=0}^k t_i}, \; k=0,1,\dots.
\end{equation}
The sequence $\theta_k\in [0,1], \; k=0,1,2,\dots$ can be seen as a {\em transformation} of the sequence $t_k \in (0,\infty),\; k=0,1,\dots$.  Indeed, the sequences $t_k > 0, \; k=0,1,\dots$ and $\theta_k\in [0,1], \; k=0,1,2,\dots$ are in one-to-one correspondence.  Observe that the output sequence $x_k,\; k=1,2,\dots$ defined via~\eqref{eq.xseq} satisfies
\[
x_{k+1} = (1-\theta_k)x_{k} + \theta_k s_k, \; k=0,1,2,\dots
\]
and a similar identity holds for $z_k,\; k=0,1,2,\dots$.

\medskip

Second, let $\D:\dom(\Psi)\times \dom(\partial (f\circ A)) \times \dom(\Psi) \times [0,1] \rightarrow \R$ be defined as follows
\begin{multline}\label{eq.D.def}
\D(x,y,s,\theta):=(f\circ A+\Psi)(x+\theta(s-x)) -
(1-\theta)(f\circ A+\Psi)(x) - \theta (f\circ A+\Psi)(s) \\
+\theta D_{f\circ A}(s,y).
\end{multline}
where $D_{f\circ A}(s,y) = f(As) - f(Ay) - \ip{g}{A(s-y)}$ for some  $g\in \partial f(Ay)$ that will always be clear from the context.
\begin{theorem}\label{thm.bregman.grad}
The iterates generated by Algorithm~\ref{algo.bregman} satisfy
\begin{multline}\label{eq.bregman.grad}
f(Ax_k) + \Psi(x_k) + \frac{\sum_{i=0}^{k-1} t_i(f^*(g_i) + \Psi^*(g_i^\Psi))
}{\sum_{i=0}^{k-1} t_i} + d_k^*(-w_k) \\
=
\frac{\sum_{i=0}^{k-1} \left(t_i \D(x_i,y_i,s_{i},\theta_i)/\theta_i -D_h(s_{i},s_{i-1})\right)}{\sum_{i=0}^{k-1} t_i}.
\end{multline}
In particular, the following perturbed Fenchel duality property holds
\begin{equation}\label{eq.pert.dual.2.2}
f(Ax_k) + \Psi(x_k) + f^*(u_k) + (\Psi+d_k)^*(-A^*u_k) \le \delta_k
\end{equation}
for
\begin{equation}
\label{eq.pertdual.grad}
\delta_k = \frac{\sum_{i=0}^{k-1}  \left(t_i\D(x_i, y_i,s_{i},\theta_i)/\theta_i - D_h(s_{i},s_{i-1})\right)}{\sum_{i=0}^{k-1} t_i}.
\end{equation}
\end{theorem}

  Theorem~\ref{thm.bregman.grad} readily yields the following generic convergence result for the iterates generated by Algorithm~\ref{algo.bregman}.

\begin{corollary}\label{cor.conv.generic}
Suppose the stepsizes $t_k$ chosen in Step 4 of Algorithm~\ref{algo.bregman} satisfy
\begin{equation}\label{eq.dc}
t_k\D(x_k, y_k,s_{k},\theta_k)/\theta_k - D_h(s_{k},s_{k-1}) \le \epsilon_k, \; k=0,1,2,\dots.
\end{equation}  
Then
\[
f(Ax_k) + \Psi(x_k) + f^*(u_k) + (\Psi+d_k)^*(-A^*u_k) \le \frac{\sum_{i=0}^{k-1}\epsilon_i}{\sum_{i=0}^{k-1} t_i}, \; k=1,2,\dots.
\]
In particular,
\begin{equation}\label{eq.conv.bound}
f(Ax_k) +\Psi(x_k) -  (f(Ax) + \Psi(x))\le  \frac{D_h( x, x_0) +\sum_{i=0}^{k-1} \epsilon_i }{\sum_{i=0}^{k-1} t_i}
\end{equation}
for all $x \in \E$.
\end{corollary}

For ease of exposition we defer the proofs of Theorem~\ref{thm.bregman.subgrad} and Theorem~\ref{thm.bregman.grad} to Section~\ref{sec.proofs}.  Section~\ref{sec.algos} below shows that the convergence rates of several popular first-order algorithms can be obtained from Theorem~\ref{thm.bregman.subgrad}, Theorem~\ref{thm.bregman.grad}, and Corollary~\ref{cor.conv.generic}.   

\section{Popular first-order algorithms}
\label{sec.algos}

We next detail how several popular first-order algorithms and their convergence properties can be seen as instances of Algorithm~\ref{algo.bregman} and the perturbed Fenchel duality results established in Section~\ref{sec.bregman}.   For ease of notation and to be consistent with the most common format  in the literature on first-order methods, throughout this section we consider the special case when $\F = \E$ and $A$ is the identity mapping. In other words, we consider the convex minimization problem
\begin{equation}\label{eq.primal.special}
\dmin_{x\in \E}  \, (f+ \Psi)(x) \Leftrightarrow \dmin_{x\in \E}  f(x) + \Psi(x).
\end{equation}
All of the developments in this section extend in a straightforward fashion to the more general problem~\eqref{eq.primal} albeit with a bit of notation overhead.

A critical detail in all of the convergence results discussed below is the choice of stepsize $t_k$ in Step 4 of Algorithm~\ref{algo.bregman}.  To that end, we will rely on the following terminology and  assumption.

\begin{definition}
 We shall say that $t > 0$ is {\em admissible} for $(s_{-}, g)$ if 
\begin{equation}\label{eq.subproblem}
\min_s \{t(\ip{g}{s} + \Psi(s)) + D_h(s,s_{-})\}
\end{equation}
is well-posed and it has a solution~$\bar s$ that satisfies the optimality conditions
\begin{equation}\label{eq.op.conds.again}
t(A^*g + g^\Psi) + \nabla h(\bar s) - \nabla h(s_-) = 0
\end{equation}
for some $g^\Psi\in \partial\Psi(\bar s)$.
\end{definition}

We will make the following assumption throughout this section.

%\newpage

\begin{assumption}\label{assump.strong} The functions $f,\Psi,h$ satisfy the following three conditions:
\begin{itemize}
\item[(a)] The function $f + \Psi$ is bounded below.
\item[(b)] 
If $t > 0$ and $(s_-,g)$ are such that the mapping 
\[
s\mapsto t(\ip{g}{s} + \Psi(s)) + D_h(s,s_{-})
\]
is bounded below then $t$ is {\em admissible} for $(s_{-}, g)$.  
\item[(c)] There exists an oracle that takes as input a tuple 
$(t,s_{-},  g)$ and either determines that $t$ is not admissible  for $(s_{-}, g)$ or computes an optimal solution of~\eqref{eq.subproblem} that satisfies the optimality conditions~\eqref{eq.op.conds}.
\end{itemize}
\end{assumption}

We bifurcate our analysis into two main cases based on the reference function $h$. Subsection \ref{sec.condgrad} discusses a {\em generalized conditional subgradient algorithm,} which is a special instance of Algorithm \ref{algo.bregman} where the reference function $h$ is identically zero.  Subsections \ref{sec.bregman.prox} and \ref{sec.bregman.fast} are devoted to {\em Bregman proximal methods} which are instances of Algorithm \ref{algo.bregman} where $h$ is a general (typically nonzero) convex differentiable reference function.

\subsection{Generalized conditional subgradient algorithm ($h\equiv 0$)}\label{sec.condgrad}

Algorithm~\ref{algo.condgrad} describes a generalized conditional subgradient algorithm for problem~\eqref{eq.primal.special}.  This format is the same one developed by Nesterov~\cite{Nest18} but without any boundedness assumption on the domain. The usual conditional gradient method~\cite{Bach15,Beck17,Jagg13}, also known as Frank-Wolfe method, corresponds to the case when $\Psi$ is the indicator function $i_C$ of some compact convex set $C\subseteq \E$ equipped with a computable linear oracle, and $f$ is differentiable. Algorithm~\ref{algo.condgrad} does not require  $\Psi$ to be of this form.  Instead, it requires nothing more than the specialization of our blanket admissibility assumption (Assumption~\ref{assump.blanket}) to the case when $h\equiv 0$.

\begin{assumption}\label{assump.condgrad}
At each iteration $k=0,1,\dots$ of Algorithm~\ref{algo.condgrad}  Step 4 is well-posed, that is, a minimizer $s_{k} \in \argmin_s\{\ip{g_k}{s} + \Psi(s)\}$ exists and is computable.
\end{assumption}  
\noindent
We note that Assumption~\ref{assump.condgrad} holds if $\dom(\Psi^*) = \E^*$ and there exists a computable oracle for $\partial \Psi^*$.  
%We also note that Assumption~\ref{assump.condgrad} is identical to Assumption~\ref{assump.blanket} when $h\equiv 0$.

\medskip

We shall also make the mild assumption that $\theta_0=1$ and $\theta_k\in (0,1),\; k=1,\dots$ in Algorithm~\ref{algo.condgrad}.  Under this assumption, it is evident that Algorithm~\ref{algo.condgrad} can be obtained as a special case of 
Algorithm~\ref{algo.bregman} by picking $ h\equiv 0, \; s_{-1}:=x_0, \; y_k = x_{k},$ and $t_k> 0, \; k=0,1,\dots$ so that~\eqref{eq.theta} holds.

\begin{algorithm}
\caption{Generalized conditional subgradient algorithm}\label{algo.condgrad}
\begin{algorithmic}[1]
	\STATE {\bf input:}  $(f,\Psi)$ and $x_{0} \in \dom(\partial f)$
	\FOR{$k=0,1,2,\dots$}
		\STATE pick $g_k \in \partial f(x_k)$
		\STATE pick $s_{k} \in \argmin_s\{\ip{g_{k}}{s} + \Psi(s)\}$
		\STATE pick $\theta_k \in [0,1]$
		\STATE let $x_{k+1} = (1-\theta_k) x_k + \theta_k s_k$
	\ENDFOR
\end{algorithmic}
\end{algorithm}
Since $h \equiv 0$ it follows that $D_h \equiv 0$ so~\eqref{eq.op.conds} becomes  $g_k + g_k^\Psi = 0$.  Hence $w_k = 0$ and $d_k^*(-w_k) = 0$ for $k=0,1,\dots$. Thus Theorem~\ref{thm.bregman.grad} implies that for $k=1,2,\dots$
%\begin{equation}
\begin{align}\label{eq.gcg}
f(x_k) + \Psi(x_k) + f^*(u_k) + \Psi^*(-u_k)
&\le \frac{\sum_{i=0}^{k-1}  t_i D(x_i,s_{i},\theta_i)/\theta_i}{\sum_{i=0}^{k-1} t_i} \notag
\\
&= \frac{\sum_{i=0}^{k-1} \sum_{j=0}^{i} t_j D(x_i,s_{i},\theta_i)}{\sum_{i=0}^{k-1} t_i},
\end{align}
%\end{equation}
where $D: \dom(\Psi)\times \dom(\Psi)\times [0,1] \rightarrow \R$ is the simplified version of  $\D$ defined as follows
\begin{equation}\label{eq.D.def.simple}
D(x,s,\theta):=D_f(x+\theta(s-x),x) + \Psi(x+\theta(s-x)) - 
(1-\theta)\Psi(x) - \theta\Psi(s). 
\end{equation}
Observe that~\eqref{eq.gcg} can be written as
\begin{equation}\label{eq.gcggap.val}
f(x_k) + \Psi(x_k) + f^*(u_k) + \Psi^*(-u_k)
\le \gcggap_k
\end{equation}
where $\gcggap_k, \; k=1,2,\dots$ is defined via $\gcggap_1 = D(x_0,s_0,1) = D_f(s_0,x_0)$ and 
\begin{equation}\label{eq.gcggap}
\gcggap_{k+1} = (1-\theta_k)\gcggap_k 
+ D(x_{k},s_k,\theta_k),
\; k=1,2,\dots.
\end{equation}
We will consider the case when $f,\Psi$ satisfy the following kind of  {\em curvature condition:} there exists a constant $M > 0$ and $\nu > 0$ such that  for  $x,s\in \dom(\Psi)$ %and  $g\in \partial f(x)$
\begin{equation}\label{eq.curv.nu}
D(x,s,\theta) \le \frac{M\theta^{1+\nu}}{1+\nu} \text{ for all } \theta \in [0,1].
\end{equation} 
The curvature condition~\eqref{eq.curv.nu} holds in particular when $\dom(\Psi)$ is bounded and $\nabla f$ satisfies  the following H\"olderian continuity assumption for some norm in $\E$: for all $x,y\in \dom(f)$
\begin{equation}\label{eq.holder}
\|\nabla f(x) - \nabla f(y)\|^* \le M\|x-y\|^\nu.
\end{equation}

The curvature condition~\eqref{eq.curv.nu} is inspired by and generalizes the following {\em curvature constant} defined by Jaggi~\cite{Jagg13} in the special case when $f$ is differentiable and $\Psi= i_C$ for some compact convex set $C\subseteq \E$: 
\[
\mathcal C_f:=\sup_{x,s\in C\atop \theta\in(0,1]}\frac{2(f(x+\theta(s-x)) - f(x) - \theta\ip{\nabla f(x)}{s-x})}{\theta^2} = \sup_{x,s\in C\atop \theta\in(0,1]}\frac{2D_f(x+\theta(s-x),x) }{\theta^2}.
\]
Indeed, $\mathcal C_f$ is the smallest constant such that for $x,s \in C$
\begin{equation}\label{eq.curv.2}
D_f(x+\theta(s-x),x) \le \frac{\mathcal C_f\theta^2}{2} \text{ for all } \theta \in [0,1].
\end{equation}
Observe that~\eqref{eq.curv.2} is identical to~\eqref{eq.curv.nu} when $\nu=1$ and $\Psi = i_C$.

The proof of our ultimate proposition for the generalized conditional subgradient method will rely on the following  {\em weighted arithmetic-geometric mean (AM-GM) inequality:} If $a,b > 0$ and $\alpha,\beta \ge0$ are such that $\alpha + \beta > 0$ then
\begin{equation}\label{eq.amgm}
a^\alpha b^\beta \le \left(\frac{\alpha a + \beta b}{\alpha + \beta}\right)^{\alpha + \beta}.
\end{equation}
This inequality is also pivotal to our analysis of the fast and universal Bregman proximal methods.

\begin{proposition}\label{prop.condgrad}
Suppose $f,\Psi$ satisfy~\eqref{eq.curv.nu} and Step 5 of Algorithm~\ref{algo.condgrad} chooses $\theta_k = \frac{1+\nu}{k+1+\nu}$. 
Then the iterates $x_k,\; k=1,2,\dots$ generated by Algorithm~\ref{algo.condgrad} satisfy
\begin{equation}\label{eq.condgrad}
f(x_k) + \Psi(x_k) + f^*(u_k) + \Psi^*(-u_k) \le M\left(\frac{1+\nu}{k+1+\nu}\right)^\nu.
\end{equation}
\end{proposition}
\begin{proof}
By~\eqref{eq.gcggap.val} it suffices to show that
\begin{equation}\label{eq.gcg.induction}
\gcggap_{k} \le  M\left(\frac{1+\nu}{k+1+\nu}\right)^\nu, \; k=1,2,\dots.
\end{equation}
We proceed by induction.  For $k=1$ inequality~\eqref{eq.curv.nu} implies that
\[
\gcggap_1 = D(x_0,s_0,1) \le \frac{M}{1+\nu} \le M\left(\frac{1+\nu}{2+\nu}\right)^\nu.
\] 
The second step follows from $(2+\nu)^\nu \le (1+\nu)^{1+\nu}$ which in turn follows from the weighted AM-GM inequality~\eqref{eq.amgm} applied to $a= 1,b= 2+\nu,\alpha = 1, \beta = \nu$.

Suppose \eqref{eq.gcg.induction} holds for $k \ge 1.$ Then~\eqref{eq.gcggap},~\eqref{eq.curv.nu}, and $\theta_k = \frac{1+\nu}{k+1+\nu}$ imply that
\begin{align*}
\gcggap_{k+1} &= \left(1-\frac{1+\nu}{k+1+\nu}\right)\gcggap_k + \frac{M}{1+\nu}\left(\frac{1+\nu}{k+1+\nu}\right)^{1+\nu} \\
&\le \frac{Mk}{k+1+\nu}\left(\frac{1+\nu}{k+1+\nu}\right)^\nu  + \frac{M}{1+\nu}\left(\frac{1+\nu}{k+1+\nu}\right)^{1+\nu}\\
&= \frac{M(k+1)(1+\nu)^\nu}{(k+1+\nu)^{1+\nu}} \\
& \le M\left(\frac{1+\nu}{k+2+\nu}\right)^\nu.
\end{align*}
The last step follows from $(k+1)(k+2+\nu)^\nu \le (k+1+\nu)^{1+\nu}$ which in turn follows from the weighted AM-GM inequality~\eqref{eq.amgm} applied to $a= k+1,b= k+2+\nu,\alpha = 1, \beta = \nu$.
\end{proof}
The above convergence rate generalizes the popular convergence rate for the conditional gradient algorithm~\cite{Beck17,Jagg13}, which relies on the curvature condition~\eqref{eq.curv.2} and sets $\theta_k = 2/(k+2)$.  The latter case corresponds to the special case $\nu=1$ and $\Psi=i_C$ in Proposition~\ref{prop.condgrad}.  

\medskip

In principle Proposition~\ref{prop.condgrad} requires knowledge of $\nu$ since it relies on the stepsize choice $\theta_k = (\nu+1)/(k+\nu+1)$.  However, it is easy to see that the $\Oh(1/k^\nu)$ convergence rate in Proposition~\ref{prop.condgrad} remains attainable without knowledge of $\nu$ if the stepsize $\theta_k \in [0,1]$ is instead chosen via the following line-search procedure
\begin{equation}\label{eq.linesearch}
\theta_k = \argmin_{\theta \in [0,1]}\{(1-\theta)\gcggap_k 
+ D(x_{k},s_k,\theta)\}.
\end{equation}
The above results are similar in spirit to the $\Oh(1/k^\nu)$ convergence results developed in~\cite[Section 2]{Nest18}.  However, unlike Proposition~\ref{prop.condgrad} which is entirely norm-independent, the results in~\cite{Nest18} have an additional dependence on the diameter of the domain of the problem which is assumed to be bounded.

The linear convergence result stated in Proposition~\ref{prop.condgrad.linear} below illustrates the power of the line-search procedure~\eqref{eq.linesearch}.  To that end, we consider the case when $f,\Psi$ satisfy the following 
type of {\em joint smoothness and strong convexity property:} there exists a constant $M>0$ such that for all $(x,s,\theta)\in\dom(\Psi)\times \dom(\Psi)\times [0,1]$ 
\begin{equation}\label{eq.smooth.conv}
D_f(x+\theta(s-x),x) \le \frac{M\theta}{1-\theta}\left[ 
(1-\theta)\Psi(x) + \theta\Psi(s) - \Psi(x+\theta(s-x))\right].%\text{ for all } \theta \in [0,1).
\end{equation} 
It is easy to see that~\eqref{eq.smooth.conv}  holds with $M=L/\mu$ when  $\nabla f$ is $L$-Lipschitz continuous and $\Psi$ is $\mu$-strongly convex on $\dom(\Psi)$ for some norm in $\E$ and some $L,\mu > 0$.  Thus~\eqref{eq.smooth.conv} can indeed be seen as a type of joint smoothness and strong convexity of $f,\Psi$.  Observe that~\eqref{eq.smooth.conv} is intrinsic to $f,\Psi$ and is norm-independent as is the following result.

\begin{proposition}\label{prop.condgrad.linear}
Suppose $f$ and $\Psi$ satisfy~\eqref{eq.smooth.conv} some constant $M > 0$
and Step 5 of Algorithm~\ref{algo.condgrad} chooses $\theta_k\in[0,1]$ via the line-search procedure~\eqref{eq.linesearch}. Then the iterates generated by Algorithm~\ref{algo.condgrad} satisfy
\[
f(x_k) + \Psi(x_k) + f^*(u_k) + \Psi^*(-u_k) \le \left(\frac{M}{M+1}\right)^{k-1} D_f(s_0,x_0).
\]
\end{proposition}
\begin{proof}  From~\eqref{eq.D.def.simple} and~\eqref{eq.smooth.conv} it follows that
for all $(x,s,\theta)\in\dom(\Psi)\times \dom(\Psi)\times [0,1]$
\begin{align*}
D(x,s,\theta) &= D_f(x+\theta(s-x),x) + \Psi(x+\theta(s-x)) -(1-\theta) \Psi(x) -\theta\Psi(s)\\
&\le \left[ 
(1-\theta)\Psi(x) + \theta\Psi(s) - \Psi(x+\theta(s-x))\right]\left(\frac{M\theta}{1-\theta} - 1\right).
\end{align*}
In particular, for $\bar \theta = 1/(M+1)\in(0,1)$ we have
\[
(1-\bar \theta)\gcggap_k 
+ D(x_{k},s_k,\bar \theta) = (1-\bar \theta)\gcggap_k.
\]
Therefore~\eqref{eq.gcggap} and the choice of $\theta_k$ via the line-search procedure~\eqref{eq.linesearch} imply that
\[
\gcggap_{k+1}  \le (1-\bar \theta)\gcggap_k = \frac{M}{M+1}\gcggap_k.
\]
Consequently
\begin{align*}
f(x_k) + \Psi(x_k) + f^*(u_k) + \Psi^*(-u_k) \le \gcggap_{k} 
&\le \left(\frac{M}{M+1}\right)^{k-1} \gcggap_1
\\&\le \left(\frac{M}{M+1}\right)^{k-1} D_f(s_0,x_0).
\end{align*}
\end{proof}
Notice that this linear convergence result does not require any boundedness assumption as in~\cite{Nest18}.  Furthermore, the line-search procedure~\eqref{eq.linesearch} does not require knowledge of $M$.

\subsection{Bregman proximal gradient and subgradient algorithm \\ ($h\not\equiv 0$)}
\label{sec.bregman.prox}

Algorithm~\ref{algo.grad} describes a template that includes both a Bregman proximal gradient and a Bregman proximal subgradient algorithm for problem~\eqref{eq.primal.special}.  Note that Algorithm~\ref{algo.grad} can be obtained as a special case of  Algorithm~\ref{algo.bregman} by picking 
$y_k = s_{k-1}$.  

\begin{algorithm}
\caption{Bregman proximal gradient and subgradient algorithm}\label{algo.grad}
\begin{algorithmic}[1]
	\STATE {\bf input:}  $(f,\Psi)$ and $s_{-1} \in \dom(\Psi)\cap\dom(h)$
	\FOR{$k=0,1,2,\dots$}
		\STATE pick $g_{k}\in\partial f(s_{k-1})$
		\STATE pick $t_k > 0$ 
		\STATE pick $s_{k} \in \argmin_s\{t_k(\ip{g_{k}}{s} + \Psi(s)) + D_h(s,s_{k-1})\}$
	\ENDFOR
\end{algorithmic}
\end{algorithm}

\subsubsection{Bregman proximal gradient}
\label{sec.grad}
Suppose Step 4 of Algorithm~\ref{algo.grad} chooses an admissible stepsize $t_k >0$ so that the following version of~\eqref{eq.dc} holds for $y_k=s_{k-1}$
\begin{equation}\label{eq.dc.simple}
t_k\D(x_k,y_k,s_{k},\theta_k)/\theta_k \le D_h(s_{k},s_{k-1}).
\end{equation}
Then Corollary~\ref{cor.conv.generic} implies that for all $x\in \E$
\begin{equation}\label{eq.conv.bound.simple}
f(x_k) +\Psi(x_k) -(f(x) + \Psi(x)) \le \frac{D_h(x, x_0)}{\sum_{i=0}^{k-1}t_i}.
\end{equation}

We will assume that Step 4 of Algorithm~\ref{algo.grad} chooses a {\em sufficiently large} admissible $t_k$ so that \eqref{eq.dc.simple} holds.  To that end, observe that for any fixed $r > 1$ a straightforward back-tracking procedure finds $t_k$  within a factor of $r$ of the largest admissible $t_k>0$ such that \eqref{eq.dc.simple} holds.  When that is the case we will say that $t_k$ is $r$-large for~\eqref{eq.dc.simple}.  We note that such a back-tracking procedure is implementable provided Assumption~\ref{assump.strong} holds.

We next consider the case when $f$ is differentiable and satisfies the following kind of smoothness relative to $h$: there exists some constant $L > 0$ such that
\begin{equation}\label{eq.smooth.1}
D_f(y,x) \le L D_h(y,x) \text{ for all } x,y \in \dom(f).
\end{equation}
Observe that~\eqref{eq.smooth.1} holds in particular when $\nabla f$ is $L$-Lipschitz and $h$ is 1-strongly convex for some norm in $\E$.  However, the relative smoothness condition~\eqref{eq.smooth.1} holds much more broadly.  The concept of relative smoothness and its close connection with first-order methods has been developed and discussed in~\cite{BausBT16,LuFN18,Tebo18,VanN16,VanN17}.  

Our analysis of the Bregman proximal convergence will rely on the following property of
the function $\D$ defined via~\eqref{eq.D.def}.  The convexity of 
$f+\Psi$ implies that
\begin{equation}\label{eq.D.ineq.2}
\D(x,y,s,\theta) \le \theta D_f(s,y).
\end{equation}

\begin{proposition}\label{prop.grad}
Suppose~\eqref{eq.smooth.1} holds and Step 4 of Algorithm~\ref{algo.grad} chooses $t_k$ that is $r$-large for~\eqref{eq.dc.simple} for some $r > 1$. Then 
the sequence $x_k, \; k=1,2,\dots$ generated by Algorithm~\ref{algo.grad} satisfies
\[
f(x_k) +\Psi(x_k) -(f(x) + \Psi(x)) \le \frac{rLD_h(x, x_0)}{k}
\]
for all $x\in \E$.
\end{proposition}
\begin{proof}
The smoothness condition~\eqref{eq.smooth.1}, Assumption~\ref{assump.strong}, inequality~\eqref{eq.D.ineq.2}, and $y_k = s_{k-1}$ imply that $t_k>0$ is admissible 
for $(s_{k-1},g_k) := (s_{k-1},\nabla f(s_{k-1}))$
and makes~\eqref{eq.dc.simple} hold if $t_k \le 1/L$.

Indeed, if $t_k \in (0,1/L]$ then~\eqref{eq.smooth.1} implies that
\begin{align*}
t_k(\ip{g_k}{s} + \Psi(s))+ D_h(s,s_{k-1}) &\ge t_k\left(\ip{g_k}{s} + \Psi(s) + D_f(s,s_{k-1})\right) \\
&= t_k(f(s) + \Psi(s) - f(s_{k-1}) + \ip{g_k}{s_{k-1}}).
\end{align*}
Hence Assumption~\ref{assump.strong} implies that any $t_k\in (0,1/L]$ is admissible.  Furthermore, inequality~\eqref{eq.D.ineq.2}, $t_k \in (0,1/L]$, $y_k = s_{k-1}$, and~\eqref{eq.smooth.1} implies that
\[
t_k\D(x_k,y_k,s_k,\theta_k)/\theta_k  \le t_k D_f(s_k,s_{k-1}) \le D_h(s_k,s_{k-1}),
\]
which is precisely~\eqref{eq.dc.simple}.

 Thus Step 4 of Algorithm~\ref{algo.grad} chooses $t_k\ge 1/(rL)$ so that~\eqref{eq.dc.simple} holds and consequently~\eqref{eq.conv.bound.simple} implies that
\[
f(x_k) +\Psi(x_k) -(f(x) + \Psi(x)) \le \frac{D_h(x, x_0)}{\sum_{i=0}^{k-1}t_i}\le \frac{rLD_h(x, x_0)}{k}
\]
for all $x\in\E$.
\end{proof}
Proposition~\ref{prop.grad} recovers the classic $\Oh(1/k)$ convergence rate of proximal gradient methods~\cite{BausBT16,BeckT09,LuFN18,Tebo18}.

\subsubsection{Bregman proximal subgradient}
\label{sec.subgrad}

Consider now the case when $f$ is not necessarily differentiable but instead the function $f$ satisfies the following kind of {\em continuity relative} to $(h,\Psi)$: there exists a constant $M > 0$ such that for all $y\in \dom(h)\cap\dom(\Psi), \; g \in \partial f(y), \; s\in \dom(h)\cap \dom(\Psi)$ and $t > 0$ 
\begin{equation}\label{eq.relcont}
\frac{M^2t^2}{2}+ t\left( \Psi(s) - \Psi(y) + \ip{g}{s-y} \right) + D_h(s,y)\ge 0. 
\end{equation}
(Recall that we assume $\dom(\Psi) \subseteq \dom(\partial f)$.)
The above concept of relative continuity coincides with the ones proposed by Lu~\cite{Lu17} and by Teboulle~\cite{Tebo18} in the following special case.  When $\Psi = i_C$ for some closed convex set $C\subseteq \E$ the relative continuity condition~\eqref{eq.relcont} holds if and only if for all $y\in \dom(h)\cap C, \; g\in \partial f(y),$ and $s\in \dom(h)\cap C$
\[
\frac{M^2t^2}{2}+ t\ip{g}{s-y} + D_h(s,y)\ge 0.
\]
To motivate  how the above condition is related to continuity, we next show that~\eqref{eq.relcont} holds for $y\in \relint(\dom(f))\cap \dom(\Psi)$ and $s\in \dom(\Psi)$ if both $f$ and $\Psi$ are Lipschitz continuous on $\dom(\Psi)$ and $h$ is 1-strongly convex on $\dom(\Psi)$ for some norm $\|\cdot\|$ in $\E$. Indeed, suppose $f$ and $\Psi$  are respectively $L_f$ and $L_\Psi$ Lipschitz continuous on $\dom(\Psi)$, and $y\in \relint(\dom(f))\cap \dom(\Psi)$.  It thus follows that for $g\in \partial f(y)$ and $s\in \dom(\Psi)$
\[
\ip{g}{s-y} \ge -L_f \|s-y\|
\]
and consequently
\[
\Psi(s) - \Psi(y) +\ip{g}{s-y} \ge -(L_\Psi+L_f) \cdot \|s-y\| \ge -M \cdot \sqrt{2D_h(s,y)}
\]
for $M:=L_\Psi+L_f$.  
Therefore for all $t>0$
\begin{align*}
\frac{M^2t^2}{2}+ t\left( \Psi(s) - \Psi(y) + \ip{g}{s-y} \right) + D_h(s,y)&\ge
\frac{M^2t^2}{2} - Mt \cdot \sqrt{2D_h(s,y)} + D_h(s,y) \\&= \left(\frac{Mt}{\sqrt{2}} - \sqrt{D_h(s,y)} \right)^2 \\
&\ge 0.
\end{align*}

Assumption~\ref{assump.strong} implies that if~\eqref{eq.relcont} holds then any $t>0$ is admissible for $(y,g)$ for all $y\in\dom(\Psi)$ and $g\in \partial f(y)$.

\begin{proposition}\label{prop.subgrad} Suppose~\eqref{eq.relcont} holds.  Then the sequence $z_k, \; k=1,2,\dots$ generated by Algorithm~\ref{algo.grad} satisfies
\begin{equation}\label{eq.subgrad.cont}
f(z_k) + \Psi(z_k) + f^*(g_k) + (\Psi+d_k)^*(-g_k)
 \le 
\frac{M^2\cdot\sum_{i=0}^{k-1} t_i^2/2}{\sum_{i=0}^{k-1} t_i}.
\end{equation}
In particular, for all $x \in \E$
\[
f(z_k) +\Psi(z_k) - (f(x) + \Psi(x)) \le 
%d_k(x, z_0) +\frac{M\cdot\sum_{i=0}^{k-1} t_i^2/2}{\sum_{i=0}^{k-1} t_i}=
\frac{D_h(x, z_0) +M^2\cdot\sum_{i=0}^{k-1} t_i^2/2}{\sum_{i=0}^{k-1} t_i},
\]
and thus if $t_i:=\mathcal C/\sqrt{k}, \; i=0,1,\dots,k-1$, then we get
\[
f(z_k) +\Psi(z_k) - (f(x) + \Psi(x)) \le \frac{D_h(x, z_0)}{\mathcal C\cdot\sqrt{k}} + \frac{\mathcal C\cdot M^2}{2\sqrt{k}}.
\]
\end{proposition}
\begin{proof}
Since Algorithm~\ref{algo.grad} chooses $y_k = s_{k-1}, \; k=0,1,2,\dots$, the relative continuity condition~\eqref{eq.relcont} implies that for $k=0,1,2,\dots$
\begin{align*}
t_k(\Psi(y_k) - \Psi(s_{k}) - \ip{g_k}{s_k-y_{k}})-D_h(s_k,s_{k-1}) \le \frac{M^2t_k^2}{2}.
\end{align*}
Thus~\eqref{eq.subgrad.cont} follows from~\eqref{eq.pert.dual.2} and~\eqref{eq.pertdual.subgrad} in Theorem~\ref{thm.bregman.subgrad}.
\end{proof}

Proposition~\ref{prop.subgrad} recovers the classic $\Oh(1/\sqrt{k})$ convergence rate of the projected gradient method, that is, the case when $\Psi = i_C$ for some closed convex $C\subseteq\R^n$ as discussed in~\cite{Lu17,Tebo18}.

\subsection{Fast and universal Bregman proximal gradient algorithm}\label{sec.bregman.fast}

Suppose $f$ is differentiable.
Algorithm~\ref{algo.fastgrad} describes a template that includes both a fast and a universal Bregman proximal gradient method.  The fast gradient method is nearly identical to the iconic fast gradient algorithms proposed by Beck and Teboulle~\cite{BeckT09} and by Nesterov~\cite{Nest13}.  The universal gradient method is nearly identical to the method proposed by Nesterov~\cite{Nest15}.  Algorithm~\ref{algo.fastgrad} relies on the sequence $\theta_k,\;k=0,1,\dots$ defined via~\eqref{eq.theta}. Algorithm~\ref{algo.fastgrad} can be obtained as a special case of  Algorithm~\ref{algo.bregman} by picking $y_{k} = (1-\theta_k)x_k + \theta_k s_{k-1}$.

\begin{algorithm}
\caption{Fast and universal Bregman proximal gradient algorithm}\label{algo.fastgrad}
\begin{algorithmic}[1]
	\STATE {\bf input:}  $(f,\Psi)$ and $x_0:=s_{-1} \in \dom(\partial f)\cap\dom(h)$
	\FOR{$k=0,1,2,\dots$}
		\STATE let $y_{k} := (1-\theta_k)x_k + \theta_k s_{k-1}$
		\STATE pick $t_k > 0$ 
		\STATE let $s_{k} \in \argmin_s\{t_k(\ip{\nabla f(y_k)}{s} + \Psi(s)) + D_h(s,s_{k-1})\}$
		\STATE let $x_{k+1} = (1-\theta_k)x_k + \theta_k s_{k}$
	\ENDFOR
\end{algorithmic}
\end{algorithm}

The main convergence properties of Algorithm~\ref{algo.fastgrad} rely on some properties of the sequences $t_k,\; k=0,1,\dots$ and $\theta_k,\; k=0,1,\dots$.  Observe that the construction~\eqref{eq.theta} implies that for $k=0,1,\dots$
\[
\frac{\theta_{k}}{t_{k}} = \frac{1}{\sum_{i=0}^k t_i}
\]
and thus
\begin{equation}\label{eq.t.theta}
\frac{\theta_{k+1}}{t_{k+1}} = (1-\theta_{k+1})\cdot\frac{\theta_{k}}{t_{k}}.
\end{equation}

We will also rely on the following lemma. For ease of exposition, we defer the proof of Lemma~\ref{lemma.nu} to Section~\ref{sec.proof.lemma}.
\begin{lemma}\label{lemma.nu}  Let $\theta_0 = 1, t_0 > 0$ and let $t_k, \theta_k$ for $ k=0,1,\dots$  be such that~\eqref{eq.t.theta} holds for $k=0,1,\dots.$  
If $\gamma \ge 1$ and $L >0$ are constants such that
\begin{equation}\label{eq.thetati}
\frac{1}{\theta_i^{\gamma-1} t_i} \le L,\; i=0,1,\dots
\end{equation}
then
\begin{equation}\label{eq.thetatk}
\frac{1}{ \sum_{i=0}^{k}t_i} = \frac{\theta_k}{ t_k} \le L\left(\frac{\gamma}{k+\gamma} \right)^{\gamma}, \; k=0,1,\dots. 
\end{equation}
\end{lemma}

Our analyses of the fast and universal Bregman proximal convergence will also rely on the following property of the function $\D$ defined via~\eqref{eq.D.def}.  The convexity of 
$\Psi$ and $f$  imply that 
\begin{align}\label{eq.D.ineq}
\D(x,y,s,\theta) &\le f(x+\theta(s-x)) - (1-\theta)f(x) - \theta f(s) + \theta D_f(s,y) \notag \\ 
&= D_f(x+\theta(s-x),y) - (1-\theta)D_f(x,y)\notag \\
& \le D_f(x+\theta(s-x),y).
\end{align}

\subsubsection{Fast Bregman proximal gradient}
\label{sec.fast.grad}
As in Algorithm~\ref{algo.grad}, suppose Step 4 of Algorithm~\ref{algo.fastgrad} chooses an admissible 
 $t_k$ so that~\eqref{eq.dc.simple} holds.  Thus~\eqref{eq.conv.bound.simple} holds as well.  We next show that under a suitable smoothness assumption, the denominator $\sum_{i=0}^{k-1}t_i$ grows faster than in Algorithm~\ref{algo.grad} thereby attaining acceleration.  To that end,
we consider the case when $f$ satisfies the following kind of smoothness relative to $h$: there exist  constants $L > 0$ and $\gamma>1$  such that for $x,s,s_{-} \in \dom(f)$ and $\theta\in [0,1]$
\begin{equation}\label{eq.smooth.2}
D_f((1-\theta)x + \theta s,(1-\theta)x + \theta s_{-}) \le  L \theta^{\gamma} D_h(s,s_{-}).
\end{equation}
The smoothness condition~\eqref{eq.smooth.2} holds in particular for $\gamma = 2$ when $h$ is 1-strongly convex and $\nabla f$ is $L$-Lipschitz continuous  for some norm in $\E$.  It also holds 
when $f$ satisfies the relative smoothness  condition~\eqref{eq.smooth.1} and $D_h$ has {\em triangle scaling exponent} $\gamma$ as defined in~\cite{HanzRX18}.

The smoothness condition~\eqref{eq.smooth.2} is evidently stronger than the relative smoothness condition~\eqref{eq.smooth.1} since  \eqref{eq.smooth.2} for $\theta = 1$ is identical to \eqref{eq.smooth.1}.  The recent developments in~\cite{Drago19} show that a condition stronger than~\eqref{eq.smooth.1} is inevitable for achieving a Bregman proximal gradient method with convergence rate faster than $\Oh(1/k)$.

\begin{proposition}\label{prop.fast}
Suppose~\eqref{eq.smooth.2} holds and 
Step 4 of Algorithm~\ref{algo.fastgrad} chooses 
$t_k$ that is $r$-large for~\eqref{eq.dc.simple} for some $r>1$.  Then the sequence $x_k, \; k=1,2,\dots$ generated by Algorithm~\ref{algo.fastgrad} satisfies
\begin{equation}\label{eq.fast.grad}
f(x_k) +\Psi(x_k) - (f(x) + \Psi(x)) 
\le \frac{ \gamma^\gamma r^\gamma  LD_h(x, x_0)}{(k+\gamma-1)^\gamma}
\end{equation}
for all $x\in \E$.
\end{proposition}
\begin{proof}
The smoothness condition~\eqref{eq.smooth.2},
Assumption~\ref{assump.strong}, and inequality~\eqref{eq.D.ineq} imply that $t_k>0$ is admissible for $(s_{k-1},g_k) := (s_{k-1},\nabla f(y_k))$ and~\eqref{eq.dc.simple} holds if $t_k\theta_k^{\gamma-1}\le 1/L$.  

Indeed, if $t_k>0$ and $t_k\theta_k^{\gamma-1} \in (0,1/L]$ then~\eqref{eq.smooth.2} and the convexity of $\Psi$ imply that
\begin{align*}
t_k&(\ip{g_k}{s} + \Psi(s))  + D_h(s,s_{k-1}) \\
&\ge t_k\left(\ip{g_k}{s} + \Psi(s) + L\theta_k^{\gamma-1} D_h(s,s_{k-1})\right) \\
& \ge \frac{t_k}{\theta_k} \left( \theta_k(\ip{g_k}{s} + \Psi(s)) + D_f( (1-\theta_k)x_k + \theta_k s,(1-\theta_k)x_k + \theta_k s_k)\right) \\
&\ge \frac{t_k}{\theta_k}((f+\Psi)((1-\theta_k)x_k + \theta_k s) - (f+\Psi)((1-\theta_k)x_k + \theta_k s_{k-1}) + \theta_k\ip{g_k}{s_{k-1}}).
\end{align*}
Hence Assumption~\ref{assump.strong} implies that any $t_k>0$ with $t_k\theta_k^{\gamma-1}\in (0,1/L]$ is admissible.  Furthermore, inequality~\eqref{eq.D.ineq}, $t_k\theta_k^{\gamma-1} \in (0,1/L]$, $y_k = (1-\theta_k)x_k + \theta_k s_{k-1}$, and~\eqref{eq.smooth.2} implies that
\begin{align*}
\frac{t_k\D(x_k,y_k,s_k,\theta_k)}{\theta_k}  
&\le  \frac{D_f((1-\theta_k)x_k+\theta_k s_k,y_{k})}{L\theta_k^\gamma} \\
&= \frac{D_f((1-\theta_k)x_k+\theta_k s_k,(1-\theta_k)x_k+\theta_k s_{k-1})}{L\theta_k^\gamma} \\
& \le D_h(s_k,s_{k-1}),
\end{align*}
which is precisely~\eqref{eq.dc.simple}.

Thus  Step 4 of Algorithm~\ref{algo.fastgrad} chooses $t_k>0$ so that~\eqref{eq.dc.simple} holds with $t_k\theta_k^{\gamma-1}\ge 1/(r^\gamma L)$.  Therefore~\eqref{eq.conv.bound.simple} and Lemma~\ref{lemma.nu} applied to  $r^\gamma L$ in lieu of $L$ imply~\eqref{eq.fast.grad}.
\end{proof}

Proposition~\ref{prop.fast} recovers the optimal convergence rate of fast proximal gradient methods~\cite{BeckT09,Nest83,Nest13}.  Indeed, all of the results discussed in~\cite{BeckT09,Nest83,Nest13} apply when $\nabla f$ is $L$-Lipchitz and $h$ is $1$-strongly convex for some norm in $\E$.  In that case the relative smoothness condition~\eqref{eq.smooth.2} holds for $\gamma=2$ and thus the iconic $\Oh(1/k^2)$ rate of convergence follows from
Proposition~\ref{prop.fast}.

\subsubsection{Universal Bregman proximal gradient}
\label{sec.universal}
Suppose $\epsilon >0$ is fixed and Algorithm~\ref{algo.fastgrad} chooses  an admissible $t_k$ so that the following version of~\eqref{eq.dc} holds:
\begin{equation}\label{eq.dc.univ}
t_k\D(x_k,y_k,s_{k},\theta_k)/\theta_k \le  D_h(s_{k},s_{k-1}) + t_k \epsilon.
\end{equation}
Then %~\eqref{eq.conv.bound}
Corollary~\ref{cor.conv.generic} implies that for all $x\in \E$
\begin{equation}\label{eq.univ}
f(x_k) + \Psi(x_k) - (f(x) + \Psi(x)) \le \frac{D_h(x, x_0)}{\sum_{i=0}^{k-1}t_i} + \epsilon.
\end{equation} 
We next consider the case when $f$ satisfies the following kind of relative smoothness condition:
there exist constants $\nu \in [0,1]$ and $M>0$  such that for $x,s,s_{-} \in \dom(f)$ and $\theta\in [0,1]$
\begin{equation}\label{eq.smooth.3}
D_f((1-\theta)x + \theta s,(1-\theta)x + \theta s_{-}) \le  \frac{2 M \theta^{1+\nu} D_h(s,s_{-})^{\frac{1+\nu}{2}}}{1+\nu}.
\end{equation}
The smoothness condition~\eqref{eq.smooth.3} holds in particular when $h$ is 1-strongly convex and $\nabla f$ satisfies the H\"olderian continuity condition~\eqref{eq.holder} for some norm in $\E$.

% following H\"olderian continuity assumption for some norm in $\E$: for all $x,y\in \dom(f)$
%\begin{equation}\label{eq.holder}
%\|\nabla f(x) - \nabla f(y)\|^* \le M\|x-y\|^\nu.
%\end{equation}
%The smoothness condition~\eqref{eq.smooth.3} also holds if the relative smoothness condition~\eqref{eq.smooth.1} holds and the Bregman distance $D_h$ satisfies the following {\em triangle scaling property}~\cite{HanzRX18} for some $\gamma \ge 1$ and all $x,s,s_{-} \in \dom(f)$ and $\theta\in [0,1]$  \[D_h((1-\theta)x + \theta s,(1-\theta)x + \theta s_{-}) \le \theta^\gamma D(s,s_{-}).\]

\begin{proposition}\label{prop.univ}
Suppose~\eqref{eq.smooth.3} holds and Step 4 of Algorithm~\ref{algo.fastgrad} chooses  $t_k$ that is $r$-large for~\eqref{eq.dc.univ} for some fixed $\epsilon > 0$ and some $r>1$.  Then the sequence $x_k, \; k=1,2,\dots$ generated by Algorithm~\ref{algo.fastgrad} satisfies
\begin{equation}\label{eq.univ.conv}
f(x_k) + \Psi(x_k) - (f(x) + \Psi(x)) \le 
\frac{2r^{\frac{1+3\nu}{1+\nu}}M^{\frac{2}{1+\nu}} D_h(x, x_0)}{\epsilon^{\frac{1-\nu}{1+\nu}} k^{\frac{1+3\nu}{1+\nu}}} + \epsilon
\end{equation}
for all $x\in\E$.
\end{proposition}
\begin{proof} 
We prove~\eqref{eq.univ.conv} for $\nu \in [0,1)$.  The case $\nu=1$ follows as a limit case.
The smoothness condition~\eqref{eq.smooth.3},
Assumption~\ref{assump.strong}, and inequality~\eqref{eq.D.ineq} imply that $t_k > 0$ is admissible and~\eqref{eq.dc.univ} holds if the following inequality holds for $s\in\dom(h)$:
\begin{equation}\label{eq.dc.univ.suff}
D_h(s,s_{k-1})^{\frac{1+\nu}{2}} \le   \frac{1+\nu}{2}\cdot\frac{\epsilon}{M\theta_k^\nu} + \frac{1+\nu}{2}\cdot\frac{D_h(s,s_{k-1})}{M\theta_k^\nu t_k}.
\end{equation}
Indeed, if $t_k > 0$ and~\eqref{eq.dc.univ.suff} holds then~\eqref{eq.smooth.3} and the convexity of $\Psi$ imply that
\begin{align*}
t_k&(\ip{g_k}{s} + \Psi(s))  + D_h(s,s_{k-1}) \\
& \ge \frac{t_k}{\theta_k} 
\left( \theta_k(\ip{g_k}{s} + \Psi(s)) + \frac{2M\theta^{1+\nu}D_h(s,s_{k-1})^{\frac{1+\nu}{2}}}{1+\nu}- \theta_k\epsilon\right) \\
& \ge \frac{t_k}{\theta_k} 
\left( \theta_k(\ip{g_k}{s} + \Psi(s)) + D_f((1-\theta_k)x_k + \theta_k s,(1-\theta_k)x_k + \theta_k s_{k-1})- \theta_k\epsilon\right) \\
&\ge \frac{t_k}{\theta_k}((f+\Psi)((1-\theta_k)x_k + \theta_k s) - (f+\Psi)((1-\theta_k)x_k + \theta_k s_{k-1}) + \theta_k\ip{g_k}{s_{k-1}} - \theta_k\epsilon).
\end{align*}
Hence Assumption~\ref{assump.strong} implies that any $t_k>0$ that satisfies~\eqref{eq.dc.univ.suff}  is admissible.  Furthermore,  $y_k = (1-\theta_k)x_k + \theta_k s_{k-1}$, together with~\eqref{eq.D.ineq},~\eqref{eq.smooth.3}, and~\eqref{eq.dc.univ.suff} imply that
\begin{align*}
\frac{t_k\D(x_k,y_k,s_k,\theta_k)}{\theta_k}  
&\le  \frac{t_kD_f((1-\theta_k)x_k+\theta_k s_k,y_{k})}{\theta_k} \\
&=  \frac{t_kD_f((1-\theta_k)x_k+\theta_k s_k,(1-\theta_k)x_k+\theta_k s_{k-1})}{\theta_k} \\
&\le \frac{2t_kM\theta_k^{\nu}D_h( s_k, s_{k-1})^{\frac{1+\nu}{2}}}{1+\nu} \\
& \le D_h(s_k,s_{k-1}) + t_k \epsilon,
\end{align*}
which is precisely~\eqref{eq.dc.univ}.

If we let $\alpha := \frac{1-\nu}{2}, \beta := \frac{1+\nu}{2}, a := \frac{1+\nu}{1-\nu}\cdot \frac{\epsilon}{M\theta_k^\nu},$ and $b:=\frac{D_h(s,s_{k-1})}{a^{\frac{1-\nu}{1+\nu}}}$ then the weighted AM-GM inequality~\eqref{eq.amgm} implies that
\[
D_h(s,s_{k-1})^{\frac{1+\nu}{2}}  = a^\alpha b^\beta \le \alpha a + \beta b  = \frac{1+\nu}{2}\cdot \frac{\epsilon}{M\theta_k^\nu} + \frac{1+\nu}{2}\cdot \frac{D_h(s,s_{k-1})}{a^{\frac{1-\nu}{1+\nu}}}.
\]
Therefore~\eqref{eq.dc.univ.suff} holds if $t_k>0$ satisfies the following inequality
\[
\frac{1}{M\theta_k^\nu t_k} \ge \frac{1}{a^{\frac{1-\nu}{1+\nu}}} = \left(\frac{1-\nu}{1+\nu}\cdot \frac{M\theta_k^\nu}{\epsilon} \right)^{\frac{1-\nu}{1+\nu}} \Leftrightarrow
\frac{1}{\theta_k^{\frac{2\nu}{1+\nu}} t_k} \ge
\left(\frac{1-\nu}{1+\nu} \right)^{\frac{1-\nu}{1+\nu}} \cdot \left(\frac{M^2}{\epsilon^{1-\nu}}\right)^{\frac{1}{1+\nu}}.
\]
Consequently, Step 4 of Algorithm~\ref{algo.fastgrad} chooses an admissible $t_k>0$ so that~\eqref{eq.dc.univ} holds with
\[
\frac{1}{\theta_k^{\frac{2\nu}{1+\nu}} t_k} \le 
r^{\frac{2\nu}{1+\nu} + 1}\left(\frac{1-\nu}{1+\nu} \right)^{\frac{1-\nu}{1+\nu}} \cdot \left(\frac{M^2}{\epsilon^{1-\nu}}\right)^{\frac{1}{1+\nu}} \le r^{\frac{1+3\nu}{1+\nu}}\left(\frac{M^2}{\epsilon^{1-\nu}}\right)^{\frac{1}{1+\nu}}.\]  
Lemma~\ref{lemma.nu} applied to $\gamma =\frac{1+3\nu}{1+\nu}$ and $L = r^{\frac{1+3\nu}{1+\nu}}\left(\frac{M^2}{\epsilon^{1-\nu}}\right)^{\frac{1}{1+\nu}} = \frac{r^{\frac{1+3\nu}{1+\nu}}M^{\frac{2}{1+\nu}}}{\epsilon^{\frac{1-\nu}{1+\nu}}}$ implies that
\[
\frac{1}{\sum_{i=0}^{k-1}t_i} \le 
\frac{r^{\frac{1+3\nu}{1+\nu}}M^{\frac{2}{1+\nu}} \gamma}{\epsilon^{\frac{1-\nu}{1+\nu}} (k-1+\gamma)^{\gamma}} \le
\frac{2r^{\frac{1+3\nu}{1+\nu}}M^{\frac{2}{1+\nu}}}{\epsilon^{\frac{1-\nu}{1+\nu}} k^{\frac{1+3\nu}{1+\nu}}}.
\]
The last step holds because $\nu\in[0,1)$ implies that $1\le \gamma < 2$. Thus~\eqref{eq.univ} yields~\eqref{eq.univ.conv}.
\end{proof}

Proposition~\ref{prop.univ} recovers the optimal convergence rate of the universal Bregman gradient algorithm established by Nesterov~\cite{Nest15}.

\bigskip

\subsubsection{Proof of Lemma~\ref{lemma.nu}}
\label{sec.proof.lemma}
We prove~\eqref{eq.thetatk} by induction on $k$.  For $k=0$ inequality~\eqref{eq.thetatk} is the same as~\eqref{eq.thetati} for $i=0$ and thus readily follows from it.  Suppose~\eqref{eq.thetatk} holds for $k \ge 0$.  Hence for some $\hat k \ge k$ we have 
\begin{equation}\label{eq.ind}
\frac{\theta_k}{t_k} = L\left(\frac{\gamma}{\hat k+\gamma} \right)^{\gamma}
\end{equation}
Putting together~\eqref{eq.t.theta},~\eqref{eq.ind}, and~\eqref{eq.thetati} applied to $i=k+1$ we get
\[
\frac{\theta_{k+1}}{t_{k+1}} = (1-\theta_{k+1})\frac{\theta_{k}}{t_{k}}
= (1-\theta_{k+1})
L\left(\frac{\gamma}{\hat k+\gamma} \right)^{\gamma}  \ge 
\frac{(1-\theta_{k+1})}{t_{k+1}\theta_{k+1}^{\gamma -1}}\left(\frac{\gamma}{\hat k+\gamma} \right)^{\gamma}.
\]  
Hence $\theta_{k+1} \ge \hat \theta$ where $\hat \theta \in (0,1)$ is the positive root of the equation
\begin{equation}\label{eq.hattheta}
\frac{\theta^{\gamma}}{1-\theta} = \left(\frac{\gamma}{\hat k+\gamma} \right)^{\gamma}.
\end{equation}
Next, observe that the weighted AM-GM inequality~\eqref{eq.amgm} applied to $a=\hat k+1,b=\hat k+1+\gamma,\alpha = 1, \beta = \gamma-1$ yields 
\[
(\hat k+1)(\hat k+1+\gamma)^{\gamma-1} \le \left(\frac{\hat k+1+(\hat k+1+\gamma)(\gamma-1)}{\gamma}\right)^\gamma = (\hat k+\gamma)^\gamma.
\]
Thus for $\theta := \frac{\gamma}{\hat k+1+\gamma}$ we get
\[
\frac{\theta^{\gamma}}{1-\theta} = \frac{\gamma^\gamma}{(\hat k+1)(\hat k+1+\gamma)^{\gamma -1}}
 \ge \frac{\gamma^\gamma}{(\hat k+\gamma)^\gamma}.
 \]
Hence the monotonicity of $\theta \mapsto  \frac{\theta^{\gamma}}{1-\theta}$ implies that the positive root $\hat \theta$ of~\eqref{eq.hattheta} satisfies $\hat \theta \le \frac{\gamma}{\hat k+1+\gamma} \le \frac{\gamma}{k+1+\gamma}.$  Finally, the inequality $\theta_{k+1} \ge \hat \theta$ together with~\eqref{eq.t.theta} and~\eqref{eq.ind} imply that
\[
\frac{\theta_{k+1}}{t_{k+1}} = (1-\theta_{k+1})\frac{\theta_{k}}{t_{k}} \le
L (1-\hat\theta)\left(\frac{\gamma}{\hat k+\gamma} \right)^{\gamma} = L\hat \theta^\gamma \le L\left(\frac{\gamma}{ k+1+\gamma} \right)^{\gamma}.
\]
Therefore~\eqref{eq.thetatk} also holds for $k+1$.
\qed

\section{Proofs of Theorem~\ref{thm.bregman.subgrad} and Theorem~\ref{thm.bregman.grad}} 
\label{sec.proofs}
\begin{proof}[Proof of Theorem~\ref{thm.bregman.subgrad}]
By rearranging terms, identity~\eqref{eq.subgrad} can be rewritten as follows
\begin{multline}\label{eq.lemma}
\frac{\sum_{i=0}^{k-1} t_i(f(Ay_{i}) + \Psi(s_{i}) + f^*(g_i) + \Psi^*(g_i^\Psi))
}{\sum_{i=0}^{k-1} t_i} + d_k^*(-w_k)
\\ 
= -  \frac{\sum_{i=0}^{k-1} (t_i \ip{A^*g_i}{s_{i}-y_i} + D_h(s_{i},s_{i-1}))}{\sum_{i=0}^{k-1} t_i}.
\end{multline}
Identity~\eqref{eq.lemma} in turn follows from 
identities~\eqref{eq.simple.identity} and~\eqref{eq.key.identity} below.  For $k=1,2,\dots$
\begin{equation}\label{eq.simple.identity}
\nabla h(s_{-1}) - \nabla h(s_{k-1}) = \sum_{i=0}^{k-1} t_i (A^*g_i+g_i^\Psi)
\end{equation}
and
\begin{multline}\label{eq.key.identity}
\sum_{i=0}^{k-1} t_i(f(Ay_i) + \Psi(s_{i}) + f^*(g_i) + \Psi^*(g_i^\Psi) ) +
\ip{\nabla h(s_{k-1}) -\nabla h(s_{-1}) }{s_{k-1}} - D_h(s_{k-1},s_{-1}) 
\\ =
-\sum_{i=0}^{k-1} (t_i \ip{A^*g_i}{s_{i}-y_i} + D_h(s_{i},s_{i-1})).
\end{multline}
Indeed,~\eqref{eq.simple.identity}, and the construction~\eqref{eq.dk} of $w_k$ and $d_k$ readily yield
\[
-w_k = \frac{1}{\sum_{i=0}^{k-1} t_i} (\nabla h(s_{k-1})-\nabla h(s_{-1})) = 
\nabla d_k(s_{k-1}).
\]
Thus 
\[
d_k^*(-w_k) = \ip{-w_k}{s_{k-1}} - d_k(s_{k-1}) = 
\frac{\ip{\nabla h(s_{k-1})-\nabla h(s_{-1})}{s_{k-1}}-D_h(s_{k-1},s_{-1})}{\sum_{i=0}^{k-1} t_i}.
\]
Therefore, after dividing~\eqref{eq.key.identity} by $\sum_{i=0}^{k-1}t_i$ we obtain~\eqref{eq.lemma}.

\medskip

We next prove~\eqref{eq.simple.identity} and~\eqref{eq.key.identity}.
Identity~\eqref{eq.simple.identity} is an immediate consequence of \eqref{eq.op.conds}. % and $s_{-1} = y_0$. 
We prove~\eqref{eq.key.identity} by induction. First, observe that $g_i\in\partial f(Ay_i)$ and $g_i^\Psi\in \partial \Psi(s_i)$ imply that
\[
f(Ay_i) + \Psi(s_{i}) +  f^*(g_i) + \Psi^*(g_i^\Psi) = \ip{A^*g_i}{y_i} + \ip{g_i^\Psi}{s_i}
\]
for $i=0,1,\dots$.  Thus from~\eqref{eq.op.conds} it follows that for $i=0,1,\dots$
\begin{equation}\label{eq.inductive.step}
t_i(f(Ay_i) + \Psi(s_{i}) +  f^*(g_i) + \Psi^*(g_i^\Psi)) + \ip{ \nabla h(s_{i}) - \nabla h(s_{i-1})}{s_{i}} 
= - t_i\ip{A^*g_i}{s_{i}-y_i}.
\end{equation}
For $k=1$ identity~\eqref{eq.key.identity} immediately follows from~\eqref{eq.inductive.step} applied to $i=0$.  Suppose~\eqref{eq.key.identity} holds for $k\ge 1$.  The three-point property of $D_h$~\cite[Lemma 3.1]{ChenT93} implies that
\begin{equation}\label{eq.three.point}
D_h(s_{k-1},s_{-1}) - D_h(s_k,s_{-1}) + \ip{\nabla h(s_{k-1}) - \nabla h(s_{-1})}{s_k-s_{k-1}}= -D_h(s_k,s_{k-1}).
\end{equation}
Thus, after adding up~\eqref{eq.key.identity},~\eqref{eq.inductive.step} applied to $i=k$, and~\eqref{eq.three.point} we get
\begin{multline*}
\sum_{i=0}^{k} t_i(f(Ay_i) + \Psi(s_{i}) + f^*(g_i) + \Psi^*(g_i^\Psi) ) +
\ip{\nabla h(s_{k}) -\nabla h(s_{-1}) }{s_{k}} - D_h(s_{k},s_{-1}) 
\\ =
-\sum_{i=0}^{k} (t_i \ip{A^*g_i}{s_{i}-y_i} + D_h(s_{i},s_{i-1})).
\end{multline*}
Hence~\eqref{eq.key.identity} also holds for $k+1$.

\medskip

To finish the proof of Theorem~\ref{thm.bregman.subgrad}
we next prove~\eqref{eq.pert.dual.2} for $\delta_k$ as in~\eqref{eq.pertdual.subgrad}.  The convexity of $f,\Psi,f^*,\Psi^*$, the constructions~\eqref{eq.xseq} and~\eqref{eq.dk} of  $z_k, u_k$ and $w_k$, and~\eqref{eq.subgrad} imply that 
\begin{multline}\label{eq.pertdual.subgrad.suffice}
f(Az_k) + \Psi(z_k) + f^*(u_k) + \Psi^*(w_k-A^*u_k) + d_k^*(-w_k)\\
 \le 
\frac{\sum_{i=0}^{k-1}  t_i(\Psi(y_{i}) - \Psi(s_{i}) - \ip{A^*g_i}{s_{i}-y_{i}}) - D_h(s_{i},s_{i-1})}{\sum_{i=0}^{k-1} t_i}.
\end{multline}
Since $(\Psi+d_k)^*(-A^*u_k) \le \Psi^*(w_k-A^*u_k) + d_k^*(-w_k)$, inequality~\eqref{eq.pert.dual.2} for $\delta_k$ as in~\eqref{eq.pertdual.subgrad} follows from~\eqref{eq.pertdual.subgrad.suffice}.

\end{proof}

\begin{proof}[Proof of Theorem~\ref{thm.bregman.grad}]
We rely on Theorem~\ref{thm.bregman.subgrad}.  Observe that 
identity~\eqref{eq.subgrad} can be rewritten as follows
\begin{multline*}%\label{eq.lemma.again}
\frac{\sum_{i=0}^{k-1} t_i(f(As_{i}) + \Psi(s_{i}) + f^*(g_i) + \Psi^*(g_i^\Psi))
}{\sum_{i=0}^{k-1} t_i} + d_k^*(-w_k)
\\ 
=  \frac{\sum_{i=0}^{k-1} (t_i D_{f\circ A}(s_{i},y_i) - D_h(s_{i},s_{i-1}))}{\sum_{i=0}^{k-1} t_i}.
\end{multline*}
Thus to show~\eqref{eq.bregman.grad} it suffices to show
\begin{multline}\label{eq.bregman.grad.v2.equiv}
f(Ax_k) + \Psi(x_k) - 
\frac{\sum_{i=0}^{k-1} t_i(f(As_{i}) + \Psi(s_{i}))
}{\sum_{i=0}^{k-1} t_i} \\
= \frac{\sum_{i=0}^{k-1} t_i\left(\D(x_i,y_i,s_{i},\theta_i)/\theta_i -  D_{f\circ A}(s_{i},y_i)\right)}{\sum_{i=0}^{k-1} t_i}.
\end{multline}
We prove~\eqref{eq.bregman.grad.v2.equiv} by induction.  The case $k=1$ readily follows as both sides in~\eqref{eq.bregman.grad.v2.equiv} equal zero when $k=1$.  Suppose~\eqref{eq.bregman.grad.v2.equiv} holds for $k$.  Since $x_{k+1} = x_k+\theta_k(s_k-x_k)$, from~\eqref{eq.D.def} it follows that
\begin{multline}\label{eq.extra}
f(Ax_{k+1})+\Psi(x_{k+1}) - (1-\theta_k)(f(Ax_k) + \Psi(x_k)) - \theta_k(f(As_k) + \Psi(s_k))\\
=\D(x_k,y_k,s_k,\theta_k) - \theta_kD_{f\circ A}(s_k,y_k) = \frac{t_k
\left(\D(x_k,y_k,s_k,\theta_k)/\theta_k - D_{f\circ A}(s_k,y_k)\right)}{\sum_{i=0}^{k} t_i}.
\end{multline}
After adding up $(1-\theta_k)$ times~\eqref{eq.bregman.grad.v2.equiv} plus~\eqref{eq.extra} we get
\begin{multline*}
f(Ax_{k+1}) + \Psi(x_{k+1}) - 
\frac{\sum_{i=0}^{k} t_i(f(As_{i}) + \Psi(s_{i}))
}{\sum_{i=0}^{k} t_i} \\
= \frac{\sum_{i=0}^{k} t_i\left( \D(x_i,y_i,s_{i},\theta_i)/\theta_i -  D_{f\circ A}(s_{i},y_i)\right)}{\sum_{i=0}^{k} t_i}.
\end{multline*}
Thus~\eqref{eq.bregman.grad.v2.equiv} also holds for $k+1$.

To finish the proof of Theorem~\ref{thm.bregman.grad}
we next prove~\eqref{eq.pert.dual.2.2} for $\delta_k$ as in~\eqref{eq.pertdual.grad}.  The convexity of $f^*,\Psi^*$ and the constructions~\eqref{eq.xseq} and~\eqref{eq.dk} of $u_k$ and $w_k$ imply that 
\begin{equation}\label{eq.final}
f^*(u_k) + \Psi^*(w_k-A^*u_k) \le \frac{\sum_{i=0}^{k-1} t_i(f^*(g_i) + \Psi^*(g_i^\Psi))
}{\sum_{i=0}^{k-1} t_i}.\end{equation}
Since $(\Psi+d_k)^*(-A^*u_k) \le \Psi^*(w_k-A^*u_k) + d_k(-w_k)$,  inequality~\eqref{eq.pert.dual.2.2} for $\delta_k$ as in~\eqref{eq.pertdual.grad} follows by putting together~\eqref{eq.bregman.grad} and~\eqref{eq.final}.
\end{proof}

\bibliographystyle{plain}
%\bibliography{FenchelBregman}
%\end{document}

\end{document}